\documentclass[12pt,reqno]{amsart}

\usepackage{geometry,amsmath,amsthm,amssymb,mathabx,latexsym}
\usepackage{url}
\usepackage[colorlinks=true,linkcolor=blue,citecolor=blue,pdfpagelabels=false]{hyperref}
\usepackage[mathscr]{eucal}

\usepackage{amsmath}
\usepackage{color}

\newcommand{\pFq}[5]{\ensuremath{{}_{#1}F_{#2} \biggl[ \genfrac{}{}{0pt}{}{#3}{#4} \biggm| {#5} \biggr]}}

\renewcommand\d{\mathrm{d}}

\let\wt\widetilde
\let\wh\widehat

\newtheorem{theorem}{Theorem}[section]

\newtheorem{lemma}[theorem]{Lemma}
\newtheorem{proposition}[theorem]{Proposition}

\theoremstyle{definition}

\numberwithin{equation}{section}

\begin{document}

\title{A modular supercongruence for $_6F_5$: An~Ap\'ery-like~story}

\author{Robert Osburn, Armin Straub and Wadim Zudilin}

\address{School of Mathematics and Statistics, University College Dublin, Belfield, Dublin 4, Ireland}
\email{robert.osburn@ucd.ie}

\address{Department of Mathematics and Statistics, University of South Alabama, 411 University Blvd N, MSPB 325, Mobile, AL 36688, USA}
\email{straub@southalabama.edu}

\address{Institute for Mathematics, Astrophysics and Particle Physics, Radboud Universiteit, PO Box 9010, 6500 GL Nijmegen, The Netherlands}
\email{w.zudilin@math.ru.nl}

\address{School of Mathematical and Physical Sciences, The University of Newcastle, Callaghan, NSW 2308, Australia}
\email{wadim.zudilin@newcastle.edu.au}


\date{\today}

\begin{abstract}
  We prove a supercongruence modulo $p^3$ between the $p$th Fourier
  coefficient of a weight 6 modular form and a truncated ${}_6F_5$-hypergeometric series.
  Novel ingredients in the proof are the comparison of two
  rational approximations to $\zeta (3)$ to produce non-trivial harmonic sum
  identities and the reduction of the resulting congruences between harmonic sums via a congruence
  between the Ap\'ery numbers and another Ap\'ery-like sequence.
\end{abstract}

\maketitle

\section{Introduction}

There has been considerable recent interest in the study of arithmetic
properties connecting $p$th Fourier coefficients of integral weight modular
forms and truncated hypergeometric series. A motivating example of this
phenomenon is the modular supercongruence \cite{kilbourn-apery06}
\begin{equation}
  \pFq{4}{3}{\frac{1}{2},\, \frac{1}{2},\, \frac{1}{2},\, \frac{1}{2}}{1,\, 1,\, 1}{1}_{p - 1} \equiv a (p) \pmod{p^3},
  \label{wt4}
\end{equation}
where $p$ is an odd prime and $a (n)$ are the Fourier coefficients of the
Hecke eigenform
\begin{equation}
  \eta (2 \tau)^4 \eta (4 \tau)^4 = \sum_{n = 1}^{\infty} a (n) q^n
  \label{eta24}
\end{equation}
of weight $4$ for the modular group $\Gamma_0 (8)$. Here and throughout, $q=e^{2\pi i \tau}$ with $\operatorname{Im}\tau>0$,
$\eta(\tau)=q^{1/24}\prod_{n=1}^\infty(1-q^n)$ is Dedekind's eta function and
\begin{equation*}
  \pFq{{n + 1}}{n}{a_0,\, a_1,\, \dots,\, a_n}{b_1,\, \dots,\, b_n}{z}_{p - 1} = \sum_{k = 0}^{p - 1} \frac{(a_0)_k
   \cdots (a_n)_k}{(b_1)_k \cdots (b_n)_k}  \frac{z^n}{n!},
\end{equation*}
with $(a)_k = a (a + 1) \cdots (a + k - 1)$, is the truncated hypergeometric
series.

Kilbourn's result \eqref{wt4} verifies one of 14 conjectural supercongruences
between truncated ${}_4 F_3$-hypergeometric series (evaluated at $1$)
corresponding to fundamental periods of the Picard--Fuchs differential equation
for Calabi--Yau manifolds of dimension $3$ and the Fourier coefficients of
modular forms of weight $4$ and varying level \cite{rv-cy}. Two more
cases have been proven in \cite{fm-hyp} and \cite{mccarthy-sc}. Moreover,
there is now a general combinatorial framework
\cite{mccarthy-harm}--\cite{mccarthy-hyp} which not only covers these 14
cases, but also the 8 cases in dimensions $1$ and $2$.
In addition, \eqref{wt4} is one of van Hamme's original 13 Ramanujan-type supercongruences (see (M.2) in \cite{vh}).
For further details on this and related topics we refer to \cite{flrst-hyp}, \cite{klmsy}, \cite{oz}, \cite{swisher}.

The purpose of this paper is to observe that a relationship akin to
\eqref{wt4} exists between a truncated ${}_6 F_5$-hypergeometric series and a
modular form of weight $6$. Our main result is the following.

\begin{theorem}
  \label{thm:main}For all odd primes $p$,
  \begin{equation}
    \pFq{6}{5}{\frac{1}{2},\, \frac{1}{2},\, \frac{1}{2},\, \frac{1}{2},\, \frac{1}{2},\,
       \frac{1}{2}}{1,\, 1,\, 1,\, 1,\, 1}{1}_{p - 1} \equiv b (p) \pmod{p^3},
  \label{wt6}
  \end{equation}
  where
  \begin{equation}
    \eta (\tau)^8 \eta (4 \tau)^4 + 8 \eta (4 \tau)^{12}  = \eta(2\tau)^{12}+32\eta(2\tau)^4\eta(8\tau)^8 = \sum_{n = 1}^{\infty} b (n) q^n
  \label{eta144}
  \end{equation}
  is the unique newform in $S_6 (\Gamma_0 (8))$.
\end{theorem}

Theorem~\ref{thm:main} is of particular practical relevance due to Weil's bounds $| b (p) | < 2 p^{5
/ 2}$, which tell us that the values of the truncated sums modulo $p^3$ are
sufficient for reconstructing the Fourier coefficients $b (p)$, and hence the
Hecke eigenform. Mortenson has further observed numerically that \eqref{wt6}
appears to hold modulo $p^5$. The technical difficulties in generalizing our approach to
verify this observation seem considerable. It would therefore be particularly interesting
whether a different approach can be found, which verifies the congruence more naturally.

The paper is organized as follows. In Section~\ref{sec:apery}, we provide
additional historic context, going back to Ap\'ery's proof of the
irrationality of $\zeta (3)$, and introduce Ap\'ery-like sequences. This
also serves to prepare for our proof of Theorem~\ref{thm:main}, which,
interestingly, involves two constructions \cite{Ne96}, \cite{Ri01}, \cite{zudilin-apery30} of rational approximations to
$\zeta (3)$ as well as a congruence between the Ap\'ery numbers and
another Ap\'ery-like sequence. This congruence is proven in Section~\ref{sec:AC}.
In Section~\ref{sec:gaussian}, we briefly review Greene's Gaussian hypergeometric series.
A result of Frechette, Ono and Papanikolas \cite{fop-traces} expresses the Fourier
coefficients $b (p)$ in terms of these finite field analogs of the classical
hypergeometric series. The Gaussian hypergeometric functions that thus arise
have been determined modulo $p^3$ in \cite{os-gaussian} in terms of sums
involving harmonic sums. In Section~\ref{sec:proof}, we reduce the resulting
congruences between sums involving harmonic numbers, then prove Theorem~\ref{thm:main}.
One of the challenging auxiliary congruences is
\begin{align}
  & \sum_{k = 0}^{\frac{p-1}2} (- 1)^k \binom{\frac{p-1}2 + k}{k}^3 \binom{\frac{p-1}2}{k}^3 \bigl(1 + 3 k
  (H_{\frac{p-1}2 + k} + H_{\frac{p-1}2 - k} - 2 H_k)\bigr) \nonumber\\
  &\quad \equiv \sum_{k = 0}^{\frac{p-1}2} \binom{\frac{p-1}2 + k}{k}^2 \binom{\frac{p-1}2}{k}^2 \pmod{p^2}.
  \label{eq:cong:intro}
\end{align}
As usual, $H_n = H_n^{(1)}$, and $H_n^{(r)}$ denote the generalized harmonic numbers
\begin{equation*}
  H_n^{(r)} = \sum_{j = 1}^n \frac{1}{j^r} .
\end{equation*}
The fact that the right-hand side of \eqref{eq:cong:intro} involves the
Ap\'ery numbers and the relation of the latter to the irrationality of
$\zeta (3)$ helped us to apply some ``irrational'' ingredients, in the
form of two different constructions of rational approximations to $\zeta (3)$,
to complete the proof. Finally, in Section~\ref{sec:ab}, we comment on
the need to certify congruences algorithmically.

\section{Historic context and Ap\'ery-like sequences}\label{sec:apery}

The Ap\'ery numbers \cite[A005259]{oeis}
\begin{equation}
  A (n) = \sum_{k = 0}^n \binom{n}{k}^2 \binom{n + k}{k}^2 \label{eq:A}
\end{equation}
rose to prominence by Ap\'ery's proof \cite{apery} of the irrationality of
$\zeta (3)$ at the end of the 1970s and were studied by number theorists in
the 1980s because of their arithmetic significance. Prominently, for instance,
Beukers conceptualized Ap\'ery's proof by realizing that the ordinary
generating function admits a parametrization by modular forms. Beukers also
established \cite{beukers-apery87} a second relation to modular forms by showing that
\begin{equation}
  A \left(\frac{p - 1}{2} \right) \equiv a (p) \pmod p,
  \label{eq:Aa:modp}
\end{equation}
where $a (n)$ are the Fourier coefficients of the Hecke eigenform \eqref{eta24}.
After some dormancy, the Ap\'ery numbers resurfaced when Ahlgren and Ono \cite{ahlgren-ono-apery}
proved Beukers' conjecture that \eqref{eq:Aa:modp} holds modulo $p^2$. In a
different direction, Beukers and Zagier \cite{zagier4} initiated the
exploration of generalizations, often referred to as Ap\'ery-like sequences,
which also arise as integral solutions to recurrence equations like
\begin{equation}
  (n + 1)^3 A (n + 1) - (2 n + 1) (17 n^2 + 17 n + 5) A (n) + n^3 A (n - 1) =
  0, \label{eq:A:rec}
\end{equation}
which is satisfied by the Ap\'ery numbers $A (n)$ and characterizes them
together with the single initial condition $A (0) = 1$.

In reducing the harmonic sums that we encounter in the proof of
Theorem~\ref{thm:main}, a crucial role is played by the sequence $C_6 (n)$,
\cite[A183204]{oeis}, where
\begin{equation}
  C_{\ell} (n) = \sum_{k = 0}^n \binom{n}{k}^{\ell} \bigl(1 - \ell k (H_k - H_{n -
  k})\bigr) . \label{eq:C:ell}
\end{equation}
The phenomenon that these sequences are integral for all positive integers
$\ell$ has been proved in \cite[Proposition~1]{kr-hyp}. For $\ell = 1, 2, 3,
4, 5$, these sequences were explicitly evaluated by Paule and Schneider
\cite{ps-harm}, who further ask whether $C_{\ell} (n)$ can be expressed as a
single sum of hypergeometric terms for $\ell \geq 6$. It turns out that
$C_6 (n)$ is one of the sporadic Ap\'ery-like sequences discovered in
\cite{cooper-sporadic} (see also \cite{zudilin-bams}), so that, for $\ell = 6$, the question of Paule and
Schneider is answered affirmatively by the following observation.

\begin{proposition}
  \label{prop:C6}The sequence $C_6 (n)$ has the binomial sum representations
  \begin{align*}
    C_6 (n) & = (- 1)^n \sum_{k = 0}^n \binom{n}{k}^2 \binom{n + k}{k}
    \binom{2 k}{n}\\
    & = \sum_{k = 0}^n (- 1)^k \binom{3 n + 1}{n - k} \binom{n + k}{k}^3,
  \end{align*}
  which make the integrality of $C_6 (n)$ transparent.
\end{proposition}

That all three sums are equal can be verified by checking that each sequence
satisfies the same three-term recursion (a variation of \eqref{eq:A:rec}).
These are recorded in \cite{ps-harm} and \cite{cooper-sporadic}, or can be
automatically derived by an algorithm such as creative telescoping. An
expression for $C_6 (n)$ as a variation of the first of the sums in
Proposition~\ref{prop:C6}, and hence the answer to the question of Paule and
Schneider, for $\ell = 6$, was already observed in \cite[Entry~17 in
Table~2]{cdd-harm}. No single-sum hypergeometric expressions for $C_\ell(n)$ are known when $\ell \geq 7$.

The following unexpected congruence between the Ap\'ery numbers $A (n)$ and
the Ap\'ery-like numbers $C_6 (n)$, from \eqref{eq:A} and \eqref{eq:C:ell},
is another ingredient in our proof of Theorem~\ref{thm:main}. It is proved in
Section~\ref{sec:AC}.

\begin{lemma}
  \label{lem:AC}For all odd primes $p$,
  \begin{equation}
    A \left(\frac{p - 1}{2} \right) \equiv C_6 \left(\frac{p - 1}{2} \right)
    \pmod{p^2} . \label{eq:AC}
  \end{equation}
\end{lemma}

We point out that suitable modular parameterizations of the generating
functions $\sum_{n = 0}^{\infty} A (n) z^n$ and $\sum_{n = 0}^{\infty} C_6(n)
z^n$ convert them into weight $2$ modular forms of level $6$ and $7$,
respectively \cite{beukers-irr} and \cite{cooper-sporadic}. We further note that the
congruence \eqref{eq:AC} is rather trivially complemented by the congruence
\begin{equation*}
  A \left(\frac{p - 1}{2} \right) \equiv D \left(\frac{p - 1}{2} \right)
   \pmod{p},
\end{equation*}
which is straightforward and is only true modulo $p$, where
\begin{equation*}
  D (n) = \sum_{n = 0}^{\infty} \binom{n}{k}^4
\end{equation*}
is another Ap\'ery-like sequence \cite[A005260]{oeis}, associated with a
modular form of weight $2$ and level $10$ (see \cite{cooper-sporadic}).

\section{Another Ap\'ery number congruence}\label{sec:AC}

This section is concerned with proving the congruence \eqref{eq:AC} of
Lemma~\ref{lem:AC} and, thereby, collecting some basic congruences involving
harmonic numbers. The form in which we will later use this congruence is
\begin{equation}
  \sum_{k = 0}^m \binom{m}{k}^2 \binom{m + k}{k}^2 \equiv \sum_{k = 0}^m
  \binom{m}{k}^6 \bigl(1 - 6 k (H_k - H_{m - k})\bigr) \pmod{p^2}.
  \label{eq:AC:H}
\end{equation}
Here, and throughout, $p$ is an odd prime and $m = (p - 1) / 2$. For our proof of the
congruence \eqref{eq:AC:H} it is however crucial to use the alternative representation
\begin{equation*}
  C_6 (n) = \sum_{k = 0}^n (- 1)^k \binom{3 n + 1}{n - k} \binom{n + k}{k}^3
\end{equation*}
for the sequence $C_6 (n)$ provided by Proposition~\ref{prop:C6}.

First, note that
\begin{equation}
  \binom{m+k}{m}
  =\frac{(m+1)_k}{k!}
  =\frac{(\frac12)_k}{k!}\,\biggl(1+\frac p2\sum_{j=0}^{k-1}\frac1{j+\frac12}+O(p^2)\biggr)
  \label{m1}
\end{equation}
and
\begin{equation}
  \binom{m}{k}
  =\frac{(-1)^k(-m)_k}{k!} =(-1)^k\,\frac{(\frac12)_k}{k!}\,\biggl(1-\frac p2\sum_{j=0}^{k-1}\frac1{j+\frac12}+O(p^2)\biggr).
  \label{m2}
\end{equation}
Now, since
\begin{equation*}
  \sum_{j=0}^{k-1}\frac1{j+\frac12}
  =\sum_{j=0}^{k-1}\frac1{j+\frac12+\frac p2}+O(p)
  =\sum_{j=0}^{k-1}\frac1{j+m+1}+O(p)
  =H_{m+k}-H_m + O(p),
\end{equation*}
we can write the expressions \eqref{m1} and \eqref{m2} in the forms
\begin{equation}
  \binom{m}{k}
  =(-1)^k\,\frac{(\frac12)_k}{k!}\,\biggl(1-\frac p2(H_{m+k}-H_m)+O(p^2)\biggr),
  \label{bm}
\end{equation}
and
\begin{align}
  \binom{m+k}{m}
  &=\frac{(\frac12)_k}{k!}\,\biggl(1+\frac p2(H_{m+k}-H_m)+O(p^2)\biggr) \nonumber \\
  &=(-1)^k\binom{m}{k} \biggl(1+\frac p2(H_{m+k}-H_{m} )+O(p^2)\biggr)^2 \nonumber \\
  &=(-1)^k\binom{m}{k} \bigl(1+p(H_{m + k}-H_m )+O(p^2)\bigr). \label{bm+k}
\end{align}

Recall that $2m=p-1$, so that
\begin{equation}
  H_{2m-k}=H_{p-1}-\sum_{j=1}^k\frac1{p-j}=\sum_{j=1}^k\biggl(\frac1j+\frac p{j^2}\biggr)+O(p^2)=H_k+pH_k^{(2)}+O(p^2).
  \label{H2mk}
\end{equation}
By swapping $k$ with $m-k$, we get
\begin{equation}
  H_{m+k}=H_{m-k}+pH_{m-k}^{(2)}+O(p^2),
\label{H2k}
\end{equation}
and, in view of the invariance of $\binom{m}{k}$ under replacing $k$ with $m-k$, we can translate formula \eqref{bm+k} to
\begin{align}
  \binom{2m-k}{m}
  &=(-1)^{m-k}\binom{m}{k} \bigl(1+p(H_{2m-k}-H_m)+O(p^2)\bigr)
  \nonumber\\
  &=(-1)^{m-k}\binom{m}{k} \bigl(1+p(H_k-H_{m})+O(p^2)\bigr),
  \label{bm2}
\end{align}
which will be useful later.

On the other hand,
\begin{align*}
  \binom{3m+1}{k}
  &=\binom{m+p}{k}=(-1)^k\frac{(-m-p)_k}{k!}
  \\
  &=(-1)^k\frac{(-m)_k}{k!}\biggl(1-p\sum_{j=0}^{k-1}\frac1{-m-p+j}+O(p^2)\biggr)
  \\
  &=\binom{m}{k} \bigl(1+p(H_m-H_{m-k})+O(p^2)\bigr),
\end{align*}
so that
\begin{equation}
  \binom{3m+1}{m-k}
  =\binom{m}{k} \bigl(1+p(H_m-H_k)+O(p^2)\bigr).
  \label{bm3}
\end{equation}
It follows from \eqref{bm+k}, \eqref{H2k} and \eqref{bm3} that
\begin{align*}
  \binom{m+k}{m}^2 \binom{m}{k}^2
  &=\binom{m}{k}^4 \bigl(1+p(H_{m-k}-H_m)+O(p^2)\bigr)^2
  \\
  &=\binom{m}{k}^4 \bigl(1+p(2H_{m-k}-2H_m)+O(p^2)\bigr)
\end{align*}
and
\begin{align*}
  &
  (-1)^k\binom{3m+1}{m-k} \binom{m+k}{m}^3
  \\ &\quad
  = \binom{m}{k}^4 \bigl(1+p(H_m-H_k)+O(p^2)\bigr)\bigl(1+p(H_{m-k}-H_m)+O(p^2)\bigr)^3
  \\ &\quad
  = \binom{m}{k}^4 \bigl(1+p(3H_{m-k}-H_k - 2H_m)+O(p^2)\bigr).
\end{align*}
It remains to use the symmetry $k\leftrightarrow m-k$ in the form
\begin{equation*}
  \sum_{k=0}^{m} \binom{m}{k}^4 H_{m-k}
  =\sum_{k=0}^{m} \binom{m}{k}^4H_k
\end{equation*}
to conclude that the desired congruence \eqref{eq:AC} is indeed true modulo $p^2$.

\section{Gaussian hypergeometric series}\label{sec:gaussian}

In the following, we discuss some preliminaries concerning Greene's Gaussian hypergeometric series \cite{g}. Let $\mathbb{F}_{p}$ denote the finite field with $p$ elements. We extend the domain of all characters $\chi$ of $\mathbb{F}^{\times}_{p}$ to $\mathbb{F}_{p}$ by defining $\chi(0)=0$. For characters $A$ and $B$ of $\mathbb{F}^{\times}_{p}$, define
\begin{equation*}
  \binom{A}{B} = \frac{B(-1)}{p} J(A, \bar{B}),
\end{equation*}
where $J(\chi, \lambda)$  denotes the Jacobi sum for $\chi$ and $\lambda$ characters of $\mathbb{F}^{\times}_{p}$. For characters $A_0,A_1,\dotsc, A_n$ and $B_1, \dotsc, B_n$ of $\mathbb{F}^{\times}_{p}$ and $x \in \mathbb{F}_{p}$, define the Gaussian hypergeometric series by
\begin{equation*}
  {_{n+1}F_n} \biggl( \begin{matrix} A_0,\, A_1,\, \dots,\, A_n \\
  B_1,\, \dots,\, B_n \end{matrix} \biggm| x \biggr)_{p}
  = \frac{p}{p-1} \sum_{\chi} \binom{A_0 \chi}{\chi} \binom{A_1 \chi}{B_1 \chi}
  \dotsm \binom{A_n \chi}{B_n \chi} \chi(x),
\end{equation*}
where the summation is over all characters $\chi$ on $\mathbb{F}^{\times}_{p}$.

We consider the case where $A_i = \phi_p$, the quadratic character, for all $i$, and $B_j= \epsilon_p$, the trivial character mod $p$, for all $j$, and write
\begin{equation*}
  {_{n+1}F_n}(x) =
  {_{n+1}F_n} \biggl( \begin{matrix} \phi_p,\, \phi_p,\, \dots,\, \phi_p \\
  \epsilon_p,\, \dots,\, \epsilon_p \end{matrix} \biggm| x \biggr)_{p}
\end{equation*}
for brevity.  By \cite{g}, $p^{n} {}_{n+1}F_{n}(x) \in \mathbb{Z}$.

For $\lambda \in \mathbb{F}_{p}$ and $\ell \geq 2$ an integer, we now define the quantities
\begin{align*}
  X_\ell (p,\lambda)
  &=\lambda^{m} \sum_{k=0}^{m} (-1)^{\ell k} \binom{m+k}{k}^\ell \binom{m}{k}^\ell \bigl(1+4\ell k(H_{m+k}-H_k) \\
  &\qquad +2\ell^2 k^2(H_{m+k}-H_k)^2-\ell k^2(H_{m+k}^{(2)}-H_k^{(2)})\bigr)\lambda^{-k}, \\
  Y_\ell(p,\lambda)
  &=\lambda^{m} \sum_{k=0}^{m} (-1)^{\ell k} \binom{m+k}{k}^\ell \binom{m}{k}^\ell \bigl(1+2\ell k(H_{m+k}- H_k)
  \\
  &\qquad -\ell k(H_{m+k}-H_{m-k})\bigr)\lambda^{-kp}, \\
  Z_\ell(p,\lambda)
  &=\lambda^{m} \sum_{k=0}^{m} {\binom{2k}k}^{2\ell}16^{-\ell k}\lambda^{-kp^2}.
\end{align*}
Here, as before, $m=(p-1)/2$.

The main result in \cite{os-gaussian} provides an expression for ${_{2\ell}F_{2\ell - 1}}$ modulo $p^3$. Precisely, we have the following.

\begin{theorem}\label{osmain}
Let $p$ be an odd prime, $\lambda \in \mathbb{F}_{p}$, and $\ell \geq 2$ be an integer. Then,
\begin{equation*}
  p^{2\ell - 1} {_{2\ell}F_{2\ell - 1}} (\lambda) \equiv
  -\bigl(p^2 X_{\ell}(p,\lambda) + pY_{\ell}(p,\lambda) + Z_{\ell}(p,\lambda)\bigr) \pmod {p^3} .
\end{equation*}
\end{theorem}

An analogous result holds for the opposite parity, that is, for $_{n+1}F_n$ when $n$ is even.

\section{Two lemmas and the proof of Theorem \ref{thm:main}} \label{sec:proof}

\begin{lemma} \label{lem1} Let $p$ be an odd prime. Then
\begin{equation*}
  X_3(p, 1) - Y_2(p, 1) \equiv (-1)^{(p-1)/2} - 1 \pmod{p}.
\end{equation*}
\end{lemma}

\begin{proof}
Consider the rational function
$$
  R(t)=R_n(t)=\frac{\prod_{j=1}^n(t-j)^2}{\prod_{j=0}^n(t+j)^2},
$$
defined for any integer $n\ge0$. Its partial fraction decomposition assumes the form
$$
  R(t)=\sum_{k=0}^n\biggl(\frac{A_k}{(t+k)^2}+\frac{B_k}{t+k}\biggr),
$$
where
$$
  A_k =\bigl(R(t)(t+k)^2\bigr)\big|_{t=-k}={\binom{n+k}k}^2{\binom nk}^2,
$$
and, on considering the logarithmic derivative of $R(t)(t+k)^2$,
\begin{align*}
  B_k
  &=\frac{\d}{\d t}\bigl(R(t)(t+k)^2\bigr)\bigg|_{t=-k} \\
  &=2\bigl(R(t)(t+k)^2\bigr)\Biggl(\sum_{j=1}^n\frac1{t-j}-\sum_{\substack{j=0\\j\ne k}}^n\frac1{t+j}\Biggr)\Bigg|_{t=-k}
  \\
  &=2A_k\,\bigl((H_k-H_{n+k})+(H_k-H_{n-k})\bigr).
\end{align*}
The related partial fraction decomposition
\begin{align*}
  tR(t)&=\sum_{k=0}^n\biggl(\frac{A_kt}{(t+k)^2}+\frac{B_kt}{t+k}\biggr)
  =\sum_{k=0}^n\biggl(\frac{A_k((t+k)-k)}{(t+k)^2}+\frac{B_k((t+k)-k)}{t+k}\biggr)
  \\
  &=\sum_{k=0}^n\biggl(-\frac{kA_k}{(t+k)^2}+\frac{A_k-kB_k}{t+k}+B_k\biggr)
\end{align*}
and the residue sum theorem imply
\begin{align*}
  &
  \sum_{k=0}^n(A_k-kB_k)
  =\sum_{\text{all finite poles}}\operatorname{Res}_{\text{pole}}tR(t)
  =-\operatorname{Res}_{t=\infty}tR(t)
  \\ &\quad
  =\text{coefficient of $s$ in Taylor's $s$-expansion of }\frac1s\,R\Bigl(\frac1s\Bigr)
  \\ &\quad
  =\text{coefficient of $s$ in Taylor's $s$-expansion of }s\,\frac{\prod_{j=1}^m(1-js)^2}{\prod_{j=0}^m(1+js)^2}
  \\ &\quad
  =1=A_0,
\end{align*}
from which $\sum_{k=1}^n (A_k-kB_k)=0$ follows. The resulting identity is then
\begin{equation}
  \sum_{k=0}^n{\binom{n+k}k}^2{\binom nk}^2\bigl(1-2k(2H_k-H_{n+k}-H_{n-k})\bigr)=1,
  \label{id1}
\end{equation}
which played a crucial role in \cite{ahlgren-ono-apery} and \cite{kilbourn-apery06}. Notice that \eqref{id1} implies
\begin{equation}
  Y_2(p,1)=1.
  \label{Y2}
\end{equation}
Equality \eqref{id1} and its derivation above follow the approach of Nesterenko from \cite{Ne96} of proving Ap\'ery's theorem
(see also \cite{zudilin-apery30}).

We can perform a similar analysis for the rational function
$$
  \wt R(t)=\wt R_n(t)=\frac{\prod_{j=1}^n(t-j)^3}{\prod_{j=0}^n(t+j)^3}
  =\sum_{k=0}^n\biggl(\frac{\wt A_k}{(t+k)^3}+\frac{\wt B_k}{(t+k)^2}+\frac{\wt C_k}{t+k}\biggr).
$$
As before, we get
\begin{align*}
  \wt A_k
  &=\bigl(\wt R(t)(t+k)^3\bigr)\big|_{t=-k}=(-1)^{n+k}{\binom{n+k}k}^3{\binom nk}^3,
  \\
  \wt B_k
  &=3\wt A_k\,(2H_k-H_{n+k}-H_{n-k}),
  \\
  \wt C_k
  &=\frac92\wt A_k\,(2H_k-H_{n+k}-H_{n-k})^2-\frac32\wt A_k\,(H_{n+k}^{(2)}-2H_k^{(2)}-H_{n-k}^{(2)})
\end{align*}
and by considering the sum of the residues of the rational functions $R(t)$, $tR(t)$ and $t^2R(t)$, we deduce that
$$
  \sum_{k=0}^n\wt C_k=\sum_{k=0}^n(\wt B_k-k\wt C_k)=0
  \qquad\text{and}\qquad
  \sum_{k=0}^n(\wt A_k-2k\wt B_k+k^2\wt C_k)=1.
$$
We only record the first and last equalities for our future use:
\begin{equation}
  \sum_{k=0}^n(-1)^k{\binom{n+k}k}^3{\binom nk}^3\bigl(3(2H_k-H_{n+k}-H_{n-k})^2-(H_{n+k}^{(2)}-2H_k^{(2)}-H_{n-k}^{(2)})\bigr)
  =0
  \label{Res1}
\end{equation}
and
\begin{multline}
  \sum_{k=0}^n(-1)^k{\binom{n+k}k}^3{\binom nk}^3\bigl(1-6k(2H_k-H_{n+k}-H_{n-k})+\tfrac92k^2(2H_k-H_{n+k}-H_{n-k})^2
  \\
  -\tfrac32k^2(H_{n+k}^{(2)}-2H_k^{(2)}-H_{n-k}^{(2)})\bigr)
  =(-1)^n.
  \label{Res3}
\end{multline}
Recall that, throughout, $m=(p-1)/2$.
Now, taking $n=m$ in \eqref{Res3} and applying $H_{m-k} \equiv H_{m+k}\pmod p$ and $H_{m-k}^{(2)} \equiv -H_{m+k}^{(2)}\pmod p$, we obtain
\begin{align}
  X_3(p,1) = \sum_{k=0}^{m} & (-1)^k \binom{m+k}{k}^3 \binom{m}{k}^3 \bigl(1-12k(H_k-H_{m + k}) \label{key} \\
  & + 18k^2(H_k-H_{m+k})^2 -3k^2(H_{m+k}^{(2)}-H_k^{(2)})\bigr)\equiv(-1)^{m} \pmod p. \nonumber
\end{align}
The result then follows after combining \eqref{Y2} with \eqref{key}.
\end{proof}

\begin{lemma} \label{lem2} Let $p$ be an odd prime. Then
\begin{equation*}
  Y_3(p,1) \equiv Z_2(p,1) \pmod{p^2}.
\end{equation*}
\end{lemma}

\begin{proof}
Consider the rational function
\begin{align*}
  \wh R(t)=\wh R_n(t)
  &=\frac{n!^2\,(2t+n)\prod_{j=1}^n(t-j)\cdot\prod_{j=1}^n(t+n+j)}{\prod_{j=0}^n(t+j)^4}
  \\
  &=\sum_{k=0}^n\biggl(\frac{\wh A_k}{(t+k)^4}+\frac{\wh B_k}{(t+k)^3}+\frac{\wh C_k}{(t+k)^2}+\frac{\wh D_k}{t+k}\biggr).
\end{align*}
Then
\begin{align*}
  \wh A_k
  &=(-1)^n((n-k)-k)\binom{n+k}n\binom{2n-k}n \binom{n}{k}^4,
  \\
  \wh B_k
  &=(-1)^n\binom{n+k}n\binom{2n-k}n{\binom nk}^4
  \bigl(2+(n-2k)\bigl(-(H_{n+k}-H_k)
  \\ &\quad
  +(H_{2n-k}-H_{n-k})-4(H_{n-k}-H_k)\bigr)\bigr).
\end{align*}

An important consequence of a hypergeometric transformation due to W.\,N.~Bailey \cite{Ba35}, \cite{Zu16}
(see also \cite{Ri01} and \cite{zudilin-apery30} for the links with rational approximations to $\zeta (3)$)
is the equality
\begin{align}
  A(n)
  =\frac12\sum_{k=0}^n\wh B_k
  &=\frac{(-1)^n}2\sum_{k=0}^n\binom{n+k}n\binom{2n-k}n{\binom nk}^4
  \nonumber\\ &\quad\times
  \bigl(2+(n-2k)(5H_k-5H_{n-k}-H_{n+k}+H_{2n-k})\bigr).
  \label{BR}
\end{align}
Now, take $n=m$ (recall that $m=(p-1)/2$) and let $b(m, k)$ denote the summand in \eqref{BR}. Note that $b(m,k)=b(m,m-k)$ and substituting of \eqref{bm+k} and \eqref{bm2} implies that
\begin{align}
  b(m,k)
  &=\binom{m}{k}^6 \bigl(1+p(H_{m+k}-H_m)+O(p^2)\bigr)
  \bigl(1+p(H_k-H_m)+O(p^2)\bigr)
  \nonumber\\ &\quad\times
  \bigl(2+(m-2k)(5H_k-5H_{m-k}-H_{m+k}+H_{2m-k})\bigr)
  \nonumber\\
  &=\binom{m}{k}^6 \bigl(1+p(H_k+H_{m-k}-2H_m)+O(p^2)\bigr)
  \nonumber\\ &\quad\times
  \bigl(2+(m-2k)\bigl(6H_k-6H_{m-k}+pH_k^{(2)}-pH_{m-k}^{(2)}+O(p^2)\bigr)\bigr)
  \nonumber\\
  &=\binom{m}{k}^6 \bigl(2+6(m-2k)(H_k-H_{m-k}) +2p(H_k+H_{m-k} -2H_m)  \nonumber \\
  &\qquad +6p(m-2k)(H_k^2-H_{m-k}^2) -12p (m-2k)(H_k-H_{m-k})H_{m} \nonumber \\
  &\qquad +p(m-2k)(H_k^{(2)}-H_{m-k}^{(2)}) +O(p^2)\bigr).
  \label{bexp}
\end{align}

Moreover, it follows from the symmetry $k \leftrightarrow m - k$ in the form
$$
  \sum_{k=0}^{m} \binom{m}{k}^6 H_{m} = \sum_{k=0}^{m} \binom{m}{k}^6 H_{m - k}
$$
as well as Lemma \ref{lem:AC}, \eqref{eq:AC:H} and \eqref{BR} that
\begin{align*}
  & \sum_{k=0}^{m} \binom{m}{k}^6 \bigl(1+3(m-2k)(H_k-H_{m-k})\bigr) \\
  &\quad =\sum_{k=0}^{m} \binom{m}{k}^6 \bigl(1-6k(H_k-H_{m-k})\bigr)
  \equiv\frac12\sum_{k=0}^{m} b(m,k)\pmod{p^2}.
\end{align*}
Substitution of the expansion \eqref{bexp} into the latter congruence results, after simplifications, in
\begin{multline}
  \sum_{k=0}^{m} \binom{m}{k}^6 \bigl(2(H_k+H_{m-k}-2H_{m})
  +6(m-2k)(H_k^2-H_{m-k}^2)
  \\
  -12(m-2k)(H_k-H_{m-k})H_{m}
  + (m-2k)(H_k^{(2)}-H_{m-k}^{(2)})\bigr)
  \equiv0\pmod p.
  \label{der1}
\end{multline}
From a different source, namely, from the equality \eqref{Res1} applied with $n=m$ and reduced modulo $p$, we obtain
\begin{equation}
  \sum_{k=0}^{m} \binom{m}{k}^6 \bigl(6(H_k-H_{m-k})^2+(H_k^{(2)}+H_{m-k}^{(2)})\bigr)
  \equiv0\pmod p.
  \label{der2}
\end{equation}
Furthermore, denote
$$
  c(m,k)=(-1)^k \binom{m+k}{k}^3 \binom{m}{k}^3 \bigl(1+3k(H_{m+k}+H_{m-k}-2H_k)\bigr),
$$
the summand of $Y_3(p,1)$. Then, with the help of \eqref{bm+k}, we obtain
\begin{align*}
  c(m,k)
  &=\binom{m}{k}^6 \bigl(1+p(H_{m+k}-H_m)+O(p^2)\bigr)^3 \nonumber \\
  &\quad \times \bigl(1+3k(H_{m+k}+H_{m-k} -2H_k)\bigr) \nonumber \\
  &=\binom{m}{k}^6 \bigl(1-6k(H_k-H_{m-k})+3p(H_{m-k}-H_m)
  \\ &\qquad
  -18pk(H_k-H_{m-k})(H_{m-k}-H_{m})+3pkH_{m-k}^{(2)}+O(p^2)\bigr)
\end{align*}
and thus
\begin{equation} \label{cc}
  \sum_{k=0}^{m} c(m,k)
  =\sum_{k=0}^{m} \wt c(m,k),
\end{equation}
where
\begin{align}
  \wt c(m,k)
  &=\frac{c(m,k)+c(m,m-k)}{2} \nonumber
  \\
  &=\binom{m}{k}^6 \bigl(1+3(m-2k)(H_k-H_{m-k}) \nonumber \\
  &\qquad +\tfrac32p(H_k+H_{m-k}-2H_{m}) -9pmH_k H_{m-k} \nonumber \\
  &\qquad -9p (m-2k)(H_k - H_{m-k})H_{m} +9p (m-k)H_k^2 \nonumber \\
  &\qquad +9pkH_{m-k}^2+\tfrac32p (m-k) H_k^{(2)}+\tfrac32pk H_{m-k}^{(2)}+O(p^2)\bigr). \label{c+c}
\end{align}
Finally, from \eqref{m1} and \eqref{m2}, we have
$$
  {\binom{2k}k}^2 2^{-4k} =\frac{(1/2)_k^2}{k!^2} \equiv(-1)^k \binom{m+k}{m} \binom{m}{k} \pmod{p^2},
$$
and so
\begin{equation} \label{ztoa}
  Z_2(p,1) \equiv A(m) \pmod{p^2}.
\end{equation}
Therefore, by \eqref{BR}, \eqref{bexp} and \eqref{cc}--\eqref{ztoa},
\begin{align*}
  Y_3(p,1) - Z_2(p,1)
  &= \sum_{k=0}^{m} c(m,k)  - \frac12 \sum_{k=0}^{m} b(m,k) \\
  &= \frac p2 \sum_{k=0}^{m} \binom{m}{k}^6 \bigl( (H_k+H_{m-k}-2H_{m}) - 18m H_k H_{m-k} \\
  &\qquad - 6(m-2k)(H_k - H_{m - k})H_{m} + (2m-k)(6H_k^2+H_k^{(2)}) \\
  &\qquad + (m+k)(6H_{m-k}^2+H_{m-k}^{(2)}) \bigr)+O(p^2).
\end{align*}
The latter sum is seen to be half of the sum in~\eqref{der1} plus $\frac32 m$ times the sum in~\eqref{der2}.
Thus, the result follows.
\end{proof}

We now prove our main result.

\begin{proof}[Proof of Theorem \ref{thm:main}]
It was conjectured by Koike and proven by Frechette, Ono and Papanikolas that the Fourier coefficients $b(p)$ of \eqref{eta144} can be represented in terms of Gaussian hypergeometric series. Specifically, we have (see Corollary~1.6 in \cite{fop-traces})
\begin{equation*}
  b(p) = - p^5 {}_6 F_5 (1) + p^4 {}_4 F_3 (1) + \bigl(1 - \phi_p(- 1)\bigr) p^2.
\end{equation*}
We now apply Theorem \ref{osmain} with $\ell=2$ and $\ell=3$, respectively, and simplify to obtain
\begin{align*}
  b(p) & \equiv p^2 \bigl(X_3(p,1) - Y_2(p,1) + 1 - (-1)^{(p-1)/2} \bigr) \\
  &\qquad + p \bigl( Y_3(p,1) - Z_2(p,1) \bigr)
  + Z_3(p,1) \pmod{p^3}.
\end{align*}
As
\begin{equation*}
  Z _3(p, 1) = \sum_{n = 0}^{(p-1)/2} \frac{(1/2)_n^6}{n!^6} \equiv \sum_{n = 0}^{p - 1} \frac{(1/2)_n^6}{n!^6} \pmod{p^6},
\end{equation*}
since the summands for $(p-1)/2 < n \leq p - 1$ are divisible by $p^6$, the result follows from Lemmas \ref{lem1} and \ref{lem2}.
\end{proof}

\section{$A\equiv B$ wanted}\label{sec:ab}

At the time of Ap\'ery's proof it was by no means trivial to verify
identities $A = B$ like the ones in Proposition~\ref{prop:C6} by verifying
that both sides, $A$ and $B$, satisfy the same recurrence. For instance, van
der Poorten's beautiful article \cite{vanderpoorten1} describes the
difficulty in checking Ap\'ery's claim that the Ap\'ery numbers $A (n)$
satisfy the recurrence \eqref{eq:A:rec}, and principally attributes to Cohen
and Zagier the clever insight to prove the claim using creative telescoping.
Since then, Wilf and Zeilberger, with subsequent support by many others, have
developed creative telescoping into a pillar of a rich computer algebraic
theory devoted to automatically proving identities between, for instance,
holonomic functions and sequences. We refer to \cite{aeqb} for a superb
introduction to these ideas. Among the more recent developments is Schneider's
work \cite{schneider-sigma07}, which extends the scope from holonomic
sequences to a class of sequences that also includes nested sums of terms
involving harmonic numbers. For instance, using Schneider's computer algebra
package SIGMA, it is routine to verify that, for all integers $n \geq 0$,
\begin{equation*}
  \sum_{k = 0}^n \binom{n}{k}^2 \binom{n + k}{k}^2 \bigl(1 - 2 k (2 H_k - H_{n +
   k} - H_{n - k})\bigr) = 1,
\end{equation*}
which we derived earlier as \eqref{id1} and which played a crucial role in
Ahlgren and Ono's proof \cite{ahlgren-ono-apery} of Beuker's conjecture as
well as Kilbourn's proof \cite{kilbourn-apery06} of the supercongruence
\eqref{wt4}.

Building on these ideas, proving our main result \eqref{wt6} modulo $p^2$,
instead of $p^3$, is much more straightforward as this corresponds to
verifying Lemma~\ref{lem2} modulo $p$ only, a task that can be performed in
many different ways (for example, using Kilbourn's strategy from
\cite[Section 4]{kilbourn-apery06}). Working modulo higher powers of $p$ is
considerably more difficult. In the course of the derivation of Theorem
\ref{thm:main} we encountered several technical difficulties that were finally
resolved by an intelligent cast of hypergeometric identities. Specifically, in
order to compute the congruence \eqref{wt6} we required the identities of
Proposition~\ref{prop:C6} as well as the equalities \eqref{id1}, \eqref{Res1},
\eqref{Res3} and \eqref{BR}, reduced modulo a suitable power of $p$. Note that
all these {\emph{identities}} can, nowadays, be easily resolved by using
computer algebraic techniques like the algorithms from \cite{aeqb} and \cite{schneider-sigma07}
mentioned above. We are, however, very restricted in this production because
certain {\emph{congruences}} (are expected to) remain not derivable this way.
For example, establishing \eqref{wt6} modulo $p^5$ (or even $p^4$) by using
appropriate intermediate identities sounds to us like a real challenge!

There is therefore a natural need for an algorithmic approach to directly
certifying congruences $A \equiv B$, say, when the terms $A$ and $B$ are
holonomic. Specifically, it would be great if such an approach could handle
congruences such as \eqref{eq:cong:intro}, or even just \eqref{eq:AC} in the
form
\begin{equation*}
  \sum_{k = 0}^n \binom{n}{k}^2 \binom{n + k}{k}^2 \equiv (- 1)^n \sum_{k =
   0}^n \binom{n}{k}^2 \binom{n + k}{k} \binom{2 k}{n} \pmod{p^2},
\end{equation*}
where $n = (p - 1) / 2$ and $p$ is an odd prime.

\section*{Acknowledgements}
The first and third authors would like to thank the organizers of the workshop ``Modular forms in String Theory'' (September 26--30, 2016)
at the Banff International Research Station, Alberta (Canada). The three authors thank the Max Planck Institute for Mathematics in Bonn (Germany), where
part of this research was performed. The third author would like to thank Ling Long for several helpful insights on links between finite
and truncated hypergeometric functions.

\end{document}